\documentclass[A4,reqno]{amsart}
\usepackage{amsthm,amsmath,amssymb,eucal}
\usepackage{geometry}
\usepackage{resizegather}
\usepackage{graphicx}
\usepackage{color}
\usepackage[normalem]{ulem}
\usepackage[latin1]{inputenc}
\usepackage{enumerate}

\newcommand{\ee}{\mathrm{e}}
\newcommand{\D}{\mathrm{d}}

\newcommand{\R}{\mathbb{R}}

\newcommand{\KK}{\mathcal{K}}

\newcommand{\MM}{\mathcal{M}}
\newcommand{\OO}{\mathcal{O}}

\newcommand{\WW}{\mathcal{W}}
\newcommand{\pd}{\partial}

\newtheorem{claim}{Claim}[section]
\newtheorem{theorem}[claim]{Theorem}
\newtheorem{proposition}[claim]{Proposition}

\newtheorem{remark}[claim]{Remark}

\title[Dirac operator spectrum in tubes and layers with a zigzag type boundary]{Dirac operator spectrum in tubes and layers with a zigzag type boundary}

\author{Pavel Exner and Markus Holzmann}

\address{{Doppler Institute for Mathematical Physics and Applied Mathematics, Czech Technical University, B\v rehov{\'a} 7, 11519 Prague, Czechia} {\rm and} {Department of Theoretical Physics, Nuclear Physics Institute, Czech Academy of Sciences, 25068 \v{R}e\v{z} near Prague, Czechia}}

\email{exner@ujf.cas.cz}

\address{Institut f\"{u}r Angewandte Mathematik, Technische Universit\"{a}t Graz, Steyrergasse 30, 8010 Graz, Austria}

\email{holzmann@math.tugraz.at}

\date{\today}

\begin{document}

\begin{abstract}
We derive a number of spectral results for Dirac operators in geometrically nontrivial regions in $\mathbb{R}^2$ and $\mathbb{R}^3$ of tube or layer shapes with a zigzag type boundary using the corresponding properties of the Dirichlet Laplacian.
\end{abstract}

\maketitle

\section{Introduction} 
\setcounter{equation}{0}

In contrast to the Schr\"odinger case, a physical motivation to study Dirac operators describing particles confined to specific regions was missing for a long time. It appeared for the first time in the 1970s in connection with the attempt to explain the quark confinement by means of the so-called bag models \cite{Bo68, DJK75}. A different, and stronger motivation came three decades later with the discovery of graphene \cite{NGM04}. Although the electrons in a graphene sheet are nonrelativistic, their behavior can be effectively described by the two-dimensional Dirac equation. Moreover, it appeared that different boundary conditions are of interest in this model depending on the way the graphene specimen is cut from a planar sheet. These applications caused an increasing mathematical interest on these types of Dirac operators, see, e.g., \cite{ALTR17, BHM20, BFSB17, LTO18, OV18} for studies on their self-adjointness and basic spectral properties.

From the mathematical point of view, relations between the shape of the region on which a given operator acts and its spectral properties belong to the most classical questions, and for operators such as the Laplacian there is a huge number of results. The corresponding problem for Dirac operators attracted attention only very recently, c.f. \cite{ABLO21, BBKO21, LO18}, and many questions remain open. A particular class of problems concerns a confinement to unbounded regions of a nontrivial geometry. For the Laplacian, or more generally, for Schr\"odinger operators, it is known that such a confinement could induce a nontrivial discrete spectrum; this problem has been thoroughly analyzed, cf. the monograph~\cite{EK15} summarizing results of many research papers.

It was noted recently that the spectrum of the Dirac operator with a particular type of boundary conditions, called `zigzag' by the physicists, and their analogue in three dimensions, can be related to the spectrum of the corresponding Dirichlet Laplacian \cite{CLMT21, Ho21, S95}. The zigzag boundary conditions in a graphene quantum dot emerge from the termination of a lattice, when the direction of the boundary is perpendicular to the bonds. The aim of this short paper is to use these results in combination with the mentioned knowledge about the Laplacian spectra in regions of tube or layer type to derive a number of new spectral results for the corresponding Dirac operators. In the next section we state in Theorem~\ref{theorem_spectrum_general} the indicated spectral correspondence and present the needed geometric preliminaries. After that we will derive our main results for the two- and three-dimensional case in Sections~\ref{s:2Dspectrum} and \ref{s:3Dspectrum}, respectively.

\section{Preliminaries} 
\setcounter{equation}{0}

\subsection{Definition of the operator and its spectrum} 

First, let us introduce the Dirac operator with zigzag boundary conditions and let us start with the two-dimensional setting. Denote by $\sigma_1, \sigma_2, \sigma_3 \in \mathbb{C}^{2 \times 2}$ the Pauli spin matrices defined by
\begin{equation} \label{def_Pauli}
  \sigma_1 = \begin{pmatrix} 0 & 1 \\1 & 0 \end{pmatrix}, \quad \sigma_2 = \begin{pmatrix} 0 & -i \\ i & 0 \end{pmatrix}, \quad \text{and} \quad \sigma_3 = \begin{pmatrix} 1 & 0 \\ 0 & -1 \end{pmatrix}.
\end{equation}
For $x  = (x_1,x_2) \in \mathbb{C}^2$ we will often use the notation $\sigma \cdot x := \sigma_1 x_1 + \sigma_2 x_2$ and, in this vein, $\sigma \cdot \nabla_2 = \sigma_1 \partial_1 + \sigma_2 \partial_2$. Next, let $\Omega \subset \mathbb{R}^2$ be an open set and let $m, c \in \mathbb{R}$ with $m \geq 0$ and $c>0$. Then the Dirac operator with zigzag boundary conditions is the differential operator $H_{m, \Omega}$ in $L^2(\Omega; \mathbb{C}^2)$ defined by
\begin{equation} \label{def_op_2D}
  \begin{split}
    H_{m, \Omega} f &= \big( -ic \sigma \cdot \nabla_2 + m c^2 \sigma_3 \big) f =  \begin{pmatrix} mc^2 &-ic (\partial_1 - i\partial_2) \\ -ic (\partial_1 + i \partial_2)& -mc^2 \end{pmatrix} f, \\
    \text{dom}\, H_{m,\Omega} &= \left\{ f = (f_1,f_2) \in L^2(\Omega; \mathbb{C}^2): (\partial_1 + i \partial_2) f_1 \in L^2(\Omega), f_2 \in H^1_0(\Omega) \right\},
  \end{split}
\end{equation}
where $H^1_0(\Omega) = \overline{C^\infty_0(\Omega)}^{\| \cdot \|_{H^1(\Omega)}}$.

In order to introduce the three-dimensional Dirac operator with zigzag type boundary conditions define the Dirac matrices $\alpha_1, \alpha_2, \alpha_3, \beta \in \mathbb{C}^{4 \times 4}$ by
\begin{equation} \label{def_alpha}
  \alpha_j = \begin{pmatrix} 0 & \sigma_j \\ \sigma_j & 0 \end{pmatrix}, \quad j \in \{1,2,3\}, \quad \text{and} \quad \beta = \begin{pmatrix} I_2 & 0 \\ 0 & -I_2 \end{pmatrix},
\end{equation}
where $\sigma_j$ are the Pauli spin matrices and $I_2$ is the $2 \times 2$-identity matrix. Similarly as above, we use for $x = (x_1, x_2, x_3) \in \mathbb{C}^3$ the notation $\alpha \cdot x = \alpha_1 x_1 + \alpha_2 x_2 + \alpha_3 x_3$ and, in this sense, $\alpha \cdot \nabla_3 = \alpha_1 \partial_1 + \alpha_2 \partial_2 + \alpha_3 \partial_3$ and also sometimes $\sigma \cdot \nabla_3 = \sigma_1 \partial_1 + \sigma_2 \partial_2 + \sigma_3 \partial_3$. Let $\Omega \subset \mathbb{R}^3$ be open. Then, choosing units such that $\hbar=1$, the three-dimensional Dirac operator with zigzag type boundary conditions is for $m,c \in \mathbb{R}$ with $m \geq 0$ and $c>0$ the differential operator acting in $L^2(\Omega; \mathbb{C}^4)$ defined by
\begin{equation} \label{def_op_3D}
  \begin{split}
    H_{m, \Omega} f &= \big( -ic \alpha \cdot \nabla_3 + m c^2 \beta \big) f =  \begin{pmatrix} mc^2 I_2 &-ic \sigma \cdot \nabla_3 \\ -ic \sigma \cdot \nabla_3 & -mc^2 I_2 \end{pmatrix} f, \\
    \text{dom}\, H_{m,\Omega} &= \left\{ f = (f_1,f_2) \in L^2(\Omega; \mathbb{C}^4): (\sigma \cdot \nabla_3) f_1 \in L^2(\Omega; \mathbb{C}^2), f_2 \in H^1_0(\Omega; \mathbb{C}^2) \right\}.
  \end{split}
\end{equation}

The basic spectral properties of $H_{m, \Omega}$ in the two- and in the three-dimensional case are summarized in the following theorem. In order to formulate it, we denote by $-\Delta^\Omega_\mathrm{D}$ the Dirichlet Laplacian in $\Omega \subset \mathbb{R}^d,\: d \in \{2,3\}$.
\begin{theorem} \label{theorem_spectrum_general}
  For any $m\ge 0$ and $c>0$ the operator $H_{m,\Omega}$ is self-adjoint. Its spectrum is
  \begin{equation*}
      \sigma(H_{m,\Omega}) = \{ m c^2 \} \cup \left\{ \pm c\sqrt{\lambda + (mc)^2}: \lambda \in \sigma(-\Delta^\Omega_\mathrm{D}) \right\}
    \end{equation*}
  and the following is true:
  \begin{itemize}
    \item[$\textup{(i)}$] $m c^2 \in \sigma_\textup{ess}(H_{m, \Omega})$.
    \item[$\textup{(ii)}$] If $m \neq 0$, then $-m c^2 \notin \sigma_\textup{p}(H_{m, \Omega})$.
    \item[$\textup{(iii)}$] Let $r=1$ for $d=2$ and $r=2$ for $d=3$. For $\lambda > 0$ one has $\pm c \sqrt{\lambda + (mc)^2} \in \sigma_\textup{p}(H_{m, \Omega})$ with multiplicity $r k$ if and only if $\lambda \in \sigma_\textup{p}(-\Delta_\mathrm{D}^\Omega)$ with multiplicity $k$.
  \end{itemize}
\end{theorem}
\begin{proof}
The statements for space dimension $d=3$ follow from~\cite[Theorem~3.4]{Ho21} noting that the operators $A_m$ in \cite{Ho21} and our $H_{m, \Omega}$ are connected by $H_{m, \Omega} = c A_{mc}$, cf. Remark~\ref{rm:dim} below.

The statements for $d=2$ follow with the same arguments as in \cite[Theorem~3.4]{Ho21}. We also refer to the proof of \cite[Theorem~2.4]{CLMT21}, where the claims are shown in the case of $C^\infty$-domains with compact boundary and $c=1$; for this one has to note that the operator $\mathcal{D}_{0,0, 2}$ in  \cite{CLMT21} is the orthogonal sum of $H_{m,\Omega}$ and $H_{m,\mathbb{R}^2 \setminus \overline{\Omega}}$.

In order to show the claims, introduce in $L^2(\Omega)$ the differential operators $\mathcal{T}_\text{min}$ and $\mathcal{T}_\text{max}$ by
\begin{equation*}
  \mathcal{T}_\textup{min} f := -i c (\partial_1 - i \partial_2) f, \qquad \text{dom}\, \mathcal{T}_\textup{min} =H^1_0(\Omega),
\end{equation*}
and
\begin{equation*}
  \mathcal{T}_\textup{max} f := -i c (\partial_1 + i \partial_2) f, \qquad \text{dom}\, \mathcal{T}_\textup{max} = \{ f \in L^2(\Omega): (\partial_1 + i \partial_2) f \in L^2(\Omega) \}.
\end{equation*}
Then as in \cite[Proposition~1]{S95} one verifies that the operator $\mathcal{T}_\text{min}$ is closed and symmetric, $\mathcal{T}_\text{max} = \mathcal{T}_\text{min}^*$,
and one can write
  \begin{equation*}
    H_{m, \Omega} = \begin{pmatrix} mc^2 & \mathcal{T}_\text{min} \\ \mathcal{T}_\text{max} & -mc^2 \end{pmatrix}.
  \end{equation*}
One finds in the same way as in \cite[Proposition~2.2 and Theorem~3.4]{Ho21}  that $\mathcal{T}_\text{max} \mathcal{T}_\text{min} = -c^2\Delta^\Omega_D$ and that $0 \in \sigma_\textup{ess}(\mathcal{T}_\text{min} \mathcal{T}_\text{max})$; cf. also \cite{S95} for similar arguments. Hence, one can apply \cite[Proposition~A.2]{Ho21} to obtain the claimed results also in the two-dimensional case.
\end{proof}

\begin{remark} \label{rm:dim}
{\rm The expression of the spectral points in \cite{CLMT21, Ho21} was stated as a mathematical result, with all the physical constants except the mass equal to one. Reintroducing the speed of light, we prefer to use the above indicated form, having in mind that the `full' expression should be
 \begin{equation} \label{dimexp}
 \pm \hbar c\sqrt{\big(\textstyle{\frac{mc}{\hbar}}\big)^2+\lambda}
 \end{equation}
which is dimensionally correct; to see that it is sufficient to realize that $\lambda$ has the dimension of inverted squared length and the energy of the free non-relativistic particle confined to the region $\Omega$ with Dirichlet boundary is $\lambda_\mathrm{nr} = \frac{\hbar^2}{2m}\lambda$.
}
\end{remark}

\begin{remark}
{\rm In three dimensions, one has $f = (f_1,f_2,f_3,f_4) \in \text{dom}\, H_{m,\Omega} \subset L^2(\Omega; \mathbb{C}^4)$, formally speaking, if the boundary conditions  $f_3|_{\partial \Omega} = f_4|_{\partial \Omega} = 0$ hold, while there are no restrictions to $f_1,f_2$. Let $\widetilde{H}_{m, \Omega}$ be the Dirac operator defined via the formal boundary conditions $f_1|_{\partial \Omega} = f_2|_{\partial \Omega} = 0$, and no restrictions are imposed to $f_3,f_4$. Then it is not difficult to show that $\widetilde{H}_{m, \Omega}$ is unitarily equivalent to $-H_{m,\Omega}$, so all the results obtained in this paper for $H_{m, \Omega}$ can be translated to corresponding results for $\widetilde{H}_{m, \Omega}$, cf. \cite[Lemma~3.2]{Ho21}. A similar argument applies in dimension two as well.}
\end{remark}

\subsection{2D geometry} \label{ss:2D}

The object of our interest are spectral properties of $H_{m,\Omega}$ for regions $\Omega$ having the form of tubes and layers. Let us describe their geometry starting from two-dimensional \emph{bent strips} of a fixed width $d=2a$. The key role is played by the strip axis, which is a curve $\Gamma$ of infinite length in $\R^2$ without angles and self-intersections, in other words, graph of a function $\Gamma:\:\R\to\R^2$; with an abuse of notation we employ the same symbol for both. We always exclude the trivial case and assume that $\Gamma$ is not a straight line.

At each point of $\Gamma$ we take the segment of the normal of length $2a$ centered at the curve; the strip $\Omega=\Omega_{\Gamma,a}$ is then the union of these segments. It can be equipped with natural curvilinear coordinates which are the arc length $s$ of $\Gamma$ and the normal distance $u$ of a strip point from the curve, so that its Cartesian coordinates are
\begin{equation} \label{bent strip coord2}
x(s,u) = \xi(s)-u\dot\eta(s), \quad y(s,u) = \eta(s)+u\dot\xi(s),
\end{equation}
where dot means derivative with respect to the arc length and the functions $\xi,\eta$ representing the parametric expression of $\Gamma$ satisfy $\dot\xi^2+\dot\eta^2=1$. In other words, $\Omega_{\Gamma,a}$ is the image of a straight strip,
 \begin{equation} \label{bentstrip}
 \Omega_{\Gamma,a}:= x(\Omega_0), \quad \Omega_0:= \R\times(-a,a).
 \end{equation}
Using the coordinate functions $\xi,\eta$ we define the \emph{signed curvature} $\gamma$ of $\Gamma$ by
$$ \gamma(s) := (\dot\eta\ddot\xi-\dot\xi\ddot\eta)(s), $$
which coincides up to the sign with the curvature understood as the inverse radius of the osculation circle to the curve, $|\gamma|=( \ddot\xi^2 +\ddot\eta^2)^{1/2}$. We assume that
 \begin{enumerate}[(a)]
 \setcounter{enumi}{0}
 \setlength{\itemsep}{1pt}
\item $\Gamma$ is $C^4$-smooth and the curvature $\gamma$ together with $\dot\gamma$ and $\ddot\gamma$ tend to zero as $|s|\to\infty$, \label{assa}
\item the map $x:\R\times(-a,a) \to \Omega_{\Gamma,a}$ is injective. \label{assb}
 \end{enumerate}
Condition \eqref{assb} means that the strip does not intersect itself; a necessary condition for that is $a\|\gamma\|_\infty<1$. The curvilinear coordinates $s,u$ are by construction locally orthogonal: if we regard $\Omega_{\Gamma,a}$ as a Riemannian manifold with a boundary, its metric tensor is diagonal with the transverse component $g_{uu}=1$ and the longitudinal one $g_{ss}$ equal to $g:=(1+u\gamma)^2$. This also means that the Jacobian of the transformation \eqref{bentstrip} is $\sqrt{g}=1+u\gamma$.

The knowledge of $\gamma$ is crucial because it allows us to reconstruct the curve, uniquely up to Euclidean transformations, and, \emph{mutatis mutandis}, the strip of a fixed halfwidth. Specifically, one uses the quantity $\beta(s_2,s_1):=\,\int_{s_1}^{s_2}\,\gamma(s)\,\D s$ which means the angle between tangent vectors at the respective points of $\Gamma$ to get for a fixed $s_0 \in \mathbb{R}$
\begin{equation} \label{curve param}
\xi(s)=\xi(s_0)+ \int_{s_0}^s\cos\beta(s_1,s_0)\,\D s_1, \quad
\eta(s)=\eta(s_0)- \int_{s_0}^s\sin\beta(s_1,s_0) \,\D s_1.
\end{equation}

Infinite strips of a fixed width are not the only regions in which the relation between the geometry and the spectrum can be studied. A case of interest, for instance, are loop-shaped strips generated by a curve $\Gamma:\,[0,L]\to \R^2$ of length $L>0$ which is closed, $\Gamma(0+)=\Gamma(L-)$. Assumption \eqref{assa} adapted to this situation says that the loop is $C^4$-smooth including, naturally, the right and left limits of the derivatives at $s=0,L$, respectively.

Returning to infinite strips, we will also consider weak deformations of straight strips of width $d>0$. For simplicity, we study only the case of a one-sided perturbation, i.e. the strip
 \begin{equation} \label{weakly def strip}
 \Omega_{\beta\!f}:= \{\, (x,y)\in\R^2:\: 0<y<d\!+\!\beta f(x)\, \}
 \end{equation}
with the deformation parameter $\beta\ge 0$ and a fixed $f\in C_0^\infty(\R)$ with $\text{supp}\, f \subset [-b,b]$.

Finally,  we may consider also other types of geometric perturbations, for instance, replacing a smooth bend by a sharply broken or polygonal shape. Another possibility is to couple different strips together. A notable example is the system of two adjacent strips of widths $d_1,\,d_2$ \emph{coupled laterally} through a `window' of width $\ell=2a$ in the common boundary, i.e. the set
\begin{equation} \label{laterally_coupled_2D}
  \Omega = \big\{ (x,y) \in \mathbb{R}^2: x \in \mathbb{R}, \,y \in (-d_2,d_1) \big\} \setminus \big( ((-\infty,-a] \cup [a,\infty)) \times \{ 0 \} \big).
\end{equation}
We denote $d:=\max\{d_1,d_2\}$, $\,D:=d_1+d_2$, and set $\varrho:=d^{-1}\,\min\{d_1,d_2\}$ as the parameter describing the asymmetry of the system.

\subsection{3D geometry} 

In three dimensions the variety of tubular regions is larger, and in addition, one can investigate spectra of Dirac particles confined to \emph{layers}. As in the previous case, let us start as before from \emph{bent tubes}. Again, if not mentioned differently we always assume that the generating curve is not a straight line.

Consider an infinite smooth curve $\Gamma$ in $\R^3$ free of self-intersections, in other words, graph of a function $\Gamma:\:\R\to\R^3$; we again employ the same symbol for both and suppose that $\Gamma$ is parametrized by its arc length, $|\dot\Gamma(s)|=1$. An important role in our considerations in the three-dimensional case is played by the Frenet triad frame $(t,n,b)$ which exists whenever $\ddot\Gamma(s)\ne 0$; it allows us to introduce curvilinear cylindrical coordinates in the vicinity of $\Gamma$ using the map
 \begin{equation} \label{curvilin}
 x(s,r,\theta):= \Gamma(s) - r\big[ n(s)\cos(\theta-\alpha(s)) + b(s)\sin(\theta-\alpha(s)) \big],
 \end{equation}
where $\alpha:\R\to\R$ is a fixed smooth function. As it is well known, the three unit vectors satisfy, as functions of $s$, the Frenet formula
 $$ 
\left(\begin{array}{ccc} \dot t \\ \dot n \\ \dot b \end{array}\right)
= \begin{pmatrix} 0 & \gamma & 0 \\ -\gamma & 0 & \tau \\ 0 & -\tau & 0 \end{pmatrix}
\left(\begin{array}{ccc} t \\ n \\ b \end{array}\right),
 $$ 
where $\gamma,\,\tau$ are the curvature and torsion of $\Gamma$, respectively; recall that the knowledge of these functions allows us again to reconstruct the curve uniquely up to Euclidean transformations. We suppose that the curve is sufficiently regular and asymptotically straight,
 \begin{enumerate}[(a)]
 \setcounter{enumi}{2}
 \setlength{\itemsep}{1pt}
\item $\Gamma$ is $C^4$-smooth, $\tau,\,\dot\tau,\,\ddot \tau$ are bounded, and the curvature $\gamma$ together with its first and second derivatives tend to zero as $|s|\to\infty$. \label{assc}
 \end{enumerate}

We are particularly interested in the coordinate frames \eqref{curvilin} satisfying the \emph{Tang condition},
 \begin{equation} \label{tang}
 \dot\alpha=\tau.
 \end{equation}
Its importance comes from the fact that it ensures that the coordinates \eqref{curvilin} are locally orthogonal. Indeed, regarding the vicinity of $\Gamma$ as a Riemannian manifold, the corresponding metric tensor is
 \begin{equation} \label{tube metric t}
(g_{ij})= \left(\begin{array}{ccc}
\left(1+r\gamma\,\cos(\theta\!-\!\alpha)\right)^2
+r^2(\tau\!-\!\dot\alpha)^2\; &\;0\; & \;r^2(\tau\!-\!\dot\alpha)
\\ 0 & 1 & 0 \\ r^2(\tau\!-\!\dot\alpha) & 0 & r^2 \end{array}
\right),
 \end{equation}
cf.~\cite[Sec.~1.3]{EK15}; we have $g:=\det (g_{ij}) = r^2\left(1+r\gamma \,\cos(\theta\!-\!\alpha)\right)^2$ and $g^{1/2}$ is the Jacobian of the map \eqref{curvilin}. In general, the Frenet triad may not exist globally, in particular, it may happen that it has one-sided limits at a point $s\in\R$ where their perpendicular parts, $(n,b)$, are not defined, but those limits do not match. Fortunately, the validity of condition \eqref{tang} is not affected if the function $\alpha$ is shifted by a constant. We assume that
 \begin{enumerate}[(a)]
 \setcounter{enumi}{3}
 \setlength{\itemsep}{1pt}
\item the coordinate system \eqref{curvilin} is \emph{Tang compatible} meaning that (i) should $\ddot\Gamma(\cdot)$ have isolated zeros, they accumulate at most at infinity, and (ii) the function $\alpha$ is piecewise continuous and such that $\dot\alpha(s)=\tau(s)$ holds whenever $\ddot\Gamma(s)\ne 0$. \label{assd}
 \end{enumerate}

To define a tube built over the curve $\Gamma$ we have to fix its cross section. We suppose that it is an open precompact set $M\subset\R^2$ containing the origin of the coordinates and we set $a:=\sup_{x\in M} |x|$; without loss of generality we may suppose that $M$ is simply connected. Using the map \eqref{curvilin}, where we identify points in $M$ with their polar coordinates, we set
 \begin{equation} \label{benttube}
 \Omega^\alpha_{\Gamma,M}:= x(\R\times M),
 \end{equation}
i.e. we identify the tube with the image of $\Omega_{\Gamma_0,M} := \R\times M$; we will drop the superscript and subscript if they are clear from the context. In order to use map \eqref{curvilin} to find spectral properties of $H_{m,\Omega}$, we have to assume in addition that
 \begin{enumerate}[(a)]
 \setcounter{enumi}{4}
 \setlength{\itemsep}{0pt}
\item the map $x$ is injective, \label{asse}
 \end{enumerate}
in other words, that the tube must not intersect itself. It is again easy to see that $a\|\gamma\|_\infty<1$ is necessary for the injectivity, however, from the global point of view it is not sufficient. In contrast to the two-dimensional situation, the existence of a locally orthogonal system of coordinates in $\Omega$ requires in view of \eqref{tube metric t} the additional assumption~\eqref{assd} which can be always satisfied provided the cross section $M$ is a disc centered at the origin giving us the freedom to choose the function $\alpha$, otherwise it is a restriction to the class of admissible tubes.

As before, the bent tubes must be neither infinite nor asymptotically straight. One can consider a loop-shaped tube build over a curve $\Gamma:\,[0,L]\to \R^3$ with an $L>0$ such that $\Gamma(0+)=\Gamma(L-)$ and the analogous relation holds for the derivatives up to the order four. Another interesting class consists of infinite curves such that their curvature and torsion are periodic functions of $s$ with the same period. In both cases assumptions \eqref{assd} and \eqref{asse} can be used again.

The importance of assumption \eqref{assd} does not mean that \emph{noncircular} tubes that do not satisfy it are not of interest; the opposite is true, just they have to be treated by other means. The case of a particular importance concerns \emph{twisted tubes}; for simplicity we restrict our attention to such tubes built over a straight line. We start again from a straight three-dimensional tube written as a Cartesian product, $\Omega_0= \R\times M$, where the cross section $M\subset\R^2$ has the same properties as before; we exclude the trivial case of a disc. Writing the element of $\R^3$ as a column, $x=(x_1,x_\perp)$, we define the twisted tube as the image of $\Omega_0$ by an appropriate map, namely
 \begin{equation} \label{twisttube}
 \Omega_\alpha := \{R_\alpha(x_1)x:\ x \in \R\times M\},
 \end{equation}
where $\alpha:\R\to\R$ is a $C^2$-smooth function with the first and second derivatives bounded on $\R$, and
$$
R_\alpha(x_1) = \left( \begin{array}{ccc}
1 & 0 & 0 \\
0 & \cos \alpha(x_1) & \sin \alpha(x_1)\\
0 & -\sin \alpha(x_1) & \cos \alpha(x_1)
\end{array} \right).
$$
Note that $\Omega_\alpha$ is nothing but the tube discussed above with $\gamma=\tau=0$ and the rotation $\alpha$ with respect to the tube axis not obeying condition \eqref{assd}.

So far we have considered tubes of a fixed cross section. Another interesting situation arises when the latter \emph{varies locally}. Consider a set-valued function $x\mapsto M_x$ which assigns to each $x\in\R$ a precompact, simply connected set $M_x\subset \R^2$ and define
 \begin{equation} \label{deftube}
\Omega := \bigcup_{x\in\R} \{ x \} \times M_x.
 \end{equation}
We assume that
 \begin{enumerate}[(a)]
 \setcounter{enumi}{5}
 \setlength{\itemsep}{0pt}
\item $\Omega$ is a local deformation of a straight tube: there is a set $M\subset \R^2$ and $x_0>0$ such that $M_x=M$ if $|x|>x_0$, \label{assf}
\item the cross section of $\Omega$ varies in a piecewise continuous manner, that is, apart of a discrete set of points, to each $x\in\R$ and $\varepsilon>0$ there is an open set $O\ni x$ such that for any $x'\in O$ the symmetric difference $(M_x\setminus M_{x'})\cup(M_{x'}\setminus M_x)$ is contained in the $\varepsilon$-neighborhood of the boundary $\partial M_x$. Moreover, the deformation is supposed to obey a global bound: there is a precompact $N\subset\R^2$ such that $M_x\subset N$ holds for all $x\in\R$. \label{assg}
 \end{enumerate}

\medskip

Up to now the generating manifold, the curve over which the tube was built, had codimension two. Let us turn to the situation when the generating manifold is of codimension one, in other words, it is a smooth surface $\Sigma$ in $\R^3$. The task of parametrizing it is now more complicated as there is no natural system of coordinates one could use. In general, one employs an atlas to describe the surface geometry, and even if it consists of a single chart, the existence of a diffeomorphism of $\Sigma$ to the plane expressed in terms of geodesic polar coordinates is not guaranteed \cite{GM69}.

For the formulation simplicity, let us assume that such coordinates exist, meaning that there is a pole $o\in\Sigma$ such that the exponential mapping $\exp_o: \mathrm{T}_o\Sigma\to\Sigma$ is a diffeomorphism; if it is not the case, everything can be rewritten in terms of suitable local charts. The `radial' coordinate lines are the geodesics emanating from $o$ and the geodesic circles connect points with the same geodesic distance from the pole. The surface $\Sigma$ is expressed by a map $p:\,\Sigma_0\to\R^3$, where $\Sigma_0:= (0,\infty)\times S^1$ is the plane with polar coordinates, $S^1$ being the unit circle; we write $q=(s,\theta)$. The tangent vectors $p_{,\mu}:=\pd p/\pd q^\mu$ are linearly independent and their cross-product defines a unit normal field~$n$ on~$\Sigma$. This allows us to define locally orthogonal coordinates in the vicinity of $\Sigma$ as the map
 \begin{equation} \label{3Dcoord}
 x(q,u):=p(q)+u n(q),
 \end{equation}
and the \emph{layer} of width~$d=2a>0$ built over the surface~$\Sigma$ as the corresponding image of the straight layer $\Omega_0:=\Sigma_0\times(-a,a)$,
 \begin{equation} \label{layer}
 \Omega:= x(\Omega_0).
 \end{equation}
We assume that the surface is not isomorphic to the $xy$ plane and that
 \begin{enumerate}[(a)]
 \setcounter{enumi}{7}
 \setlength{\itemsep}{0pt}
\item the map \eqref{3Dcoord} is $C^3$-smooth and injective on $\Omega_0$, i.e. the layer $\Omega$ does not intersect itself. In particular, locally the injectivity requires $\,a<\rho_m:= \left(\max\left\{\|k_1\|_\infty, \|k_2\|_\infty\right\}\right)^{-1}$, where $k_j$ are the principal curvatures mentioned below, which thus have to be uniformly bounded, $\|k_j\|_\infty <\infty\:$ for $j=1,2$. \label{assh}
 \end{enumerate}
Metric properties of $\Omega$ are derived from those of the generating surface $\Sigma$. Its metric tensor, $g_{\mu\nu}:= p_{,\mu} \cdot p_{,\nu}$, has in the geodesic polar coordinates a diagonal form, $(g_{\mu\nu})=\mathrm{diag}(1,r^2)$, where $r^2\equiv g:=\det(g_{\mu\nu})$ is the square of the Jacobian of the exponential mapping which satisfies the classical Jacobi equation
 \begin{equation} \label{Jacobi}
 \ddot{r}(s,\theta)+K(s,\theta)\,r(s,\theta)=0
 \quad \textrm{with}\quad
 r(0,\theta)=1\!-\!\dot{r}(0,\theta)=0,
 \end{equation}
where $\dot r$ denotes the partial derivative of $r$ with respect to $s$. The Gauss curvature $K$ appearing in (\ref{Jacobi}) together with the mean curvature $M$ are determined in the usual way: the second fundamental form $h_{\mu\nu}:= -n_{,\mu} \cdot p_{,\nu}$ gives rise to the Weingarten tensor $h_\mu^{\:\:\nu}:= h_{\mu\rho}g^{\rho\nu}$, which in turn defines the said two curvatures by $K:=\det(h_\mu^{\:\:\nu})$ and $M:={1\over 2}\mathrm{tr} (h_\mu^{\:\:\nu})$. What is important for us are the corresponding global quantities obtained by integrating with respect to the invariant surface element, $\D\sigma:=g^{1/2}\D q$, the total Gauss curvature $\KK$ and the quantity $\MM$, defined respectively by
 $$ 
 \KK:=\int_{\Sigma_0} K(q)\, \D\sigma\,,
 \quad \MM:= \left(\int_{\Sigma_0} M(q)^2\, \D\sigma
 \right)^{1/2}\,.
 $$ 
The latter always exists, being possibly infinite, while $\KK$ requires the integral to make sense which is matter of assumption, cf.~\eqref{assj} below. Recall also that the eigenvalues of the Weingarten map matrix are the principal curvatures $k_1,k_2$ through which the local Gauss and mean curvatures are expressed as $K=k_1 k_2$ and $M=\frac{1}{2}(k_1+k_2)$, respectively.

To describe layer classes we will be interested in, we adopt a couple of geometric assumptions:
 \begin{enumerate}[(a)]
 \setcounter{enumi}{8}
 \setlength{\itemsep}{0pt}
\item The layer is asymptotically planar, that is, $\,K(x),\,M(x) \to 0\,$ for $|x| \rightarrow \infty$. \label{assi}
 \end{enumerate}
If we have the geodesic polar coordinates, this means that $\,K(s,\theta),\,M(s,\theta) \to 0$ holds as $\,s\to\infty$, however assumption \eqref{assi} makes sense also if the atlas is more complicated, because the geodetic distance from a fixed point is well defined. Furthermore, we suppose that
 \begin{enumerate}[(a)]
 \setcounter{enumi}{9}
 \setlength{\itemsep}{0pt}
\item the total Gauss curvature exists, $\,K\in L^1(\Sigma_0,\D\sigma)$. \label{assj}
 \end{enumerate}
As in the case of tubes, the parametrization using coordinates \eqref{3Dcoord} can also be used to describe other curved layers of a fixed width, such as those built over a compact surface without a boundary, periodically curved layers, etc.

In addition to curved fixed-width layers we can consider flat ones with local deformations such as, for instance,
 \begin{equation} \label{bulgedlayer}
 \Omega_f:= \{x=(y,z):\, y\in\R^2,\, 0<z<d+f(y)\},
 \end{equation}
where $f:\R^2\to[0,\infty)$ is a bounded function of a compact support, or layers with a two-sided bulge. We have other cases of interest, an important one concerns \emph{laterally coupled layers}. In analogy with the two-dimensional case, by that we mean two adjacent flat layers of the widths $d_1,\,d_2>0$; we suppose that that their common boundary contains a `window' in the form of open set $\WW\subset\R^2$, meaning that $\Omega$ has the form
\begin{equation} \label{laterally_coupled_3D}
  \Omega = \big\{ (x,y, z) \in \mathbb{R}^3: x,y \in \mathbb{R},\, z \in (-d_2,d_1) \big\} \setminus \big( (\mathbb{R}^2 \setminus \WW) \times \{ 0 \} \big).
\end{equation}

\section{The geometrically induced spectrum in two dimensions}\label{s:2Dspectrum}
\setcounter{equation}{0}

After these preliminaries we can describe relations between the spectrum of the operator $H_{m,\Omega}$ and the geometry of the region $\Omega$ that supports it. We consider first the two-dimensional situations and the essential spectrum.

\begin{theorem} \label{thm:spectess2D}
We have $\sigma_\mathrm{ess}(H_{m,\Omega}) = (-\infty,-\epsilon_\mathrm{t}] \cup \{mc^2\} \cup [\epsilon_\mathrm{t},\infty)$, where
 \begin{enumerate}[(i)]
  \setlength{\itemsep}{0pt}
\item $\epsilon_\mathrm{t} = c\sqrt{m^2c^2+\big(\frac{\pi}{d}\big)^2}$ if $\Omega$ is a bent strip of width $d=2a$ satisfying assumptions~\eqref{assa} and \eqref{assb}.
\item The same is true if $\Omega$ is the weakly deformed strip in~\eqref{weakly def strip}.
\item For laterally coupled layers as in~\eqref{laterally_coupled_2D} of the widths $d_1,\,d_2>0$ we have $\epsilon_\mathrm{t} = c\sqrt{m^2c^2+\big(\frac{\pi}{d}\big)^2}$, where $d=\max\{d_1,d_2\}$.
 \end{enumerate}
\end{theorem}
\begin{proof}
In view of Theorem~\ref{theorem_spectrum_general}, the essential spectrum is determined by that of the Dirichlet Laplacian in the respective situations, cf.~Proposition~1.1.1 and Theorems~1.4 and~1.5 in \cite{EK15}.
\end{proof}

However, a nontrivial geometry of $\Omega$ can give rise to a nonvoid discrete spectrum of $H_{m,\Omega}$ which is by Theorem~\ref{theorem_spectrum_general} mirror-symmetric, consisting of eigenvalues $\pm\lambda_j$, supposed to be arranged in the ascending and descending order in the positive and negative part, respectively, multiplicity included.
\begin{theorem} \label{thm:bentstrip}
Let $\Omega$ be a bent strip of halfwidth $a$ satisfying assumptions~\eqref{assa} and \eqref{assb}, then $\sigma_\mathrm{disc}(H_{m,\Omega}) \ne\emptyset$. Furthermore, let $\{\Omega_\beta\}$ be a family of such strips with the curvature equal to $\beta\gamma$ for a fixed function $\gamma$ consistent with \eqref{assa} and~\eqref{assb} such that $\gamma, \dot\gamma, |\ddot\gamma|^{1/2} \in L^2(\mathbb{R}, |s| \D s)$; then for all $\beta$ small enough $H_{m,\Omega_\beta}$ has just two simple discrete eigenvalues $\pm \lambda_1(\beta)$ such that
 \begin{equation} \label{weakbentstrip}
 \lambda_1(\beta) = c\sqrt{m^2c^2+\epsilon(\beta)},
 \end{equation}
where
 $$
 \sqrt{\big(\textstyle{\frac{\pi}{2a}}\big)^2\!-\!\epsilon(\beta)} = {\beta^2\over 8}\,\biggl\lbrace\, \|\gamma\|^2
 -\,{1\over 2}\, \sum_{n=2}^{\infty}\, (\chi_n,u\chi_1)^2  \varrho_n \int_{\R^2} \dot\gamma(s)\,
 \ee^{-\varrho_n|s-s'|} \dot\gamma(s')\, \D s\,\D s'\, \biggr\rbrace+ \OO(\beta^3),
 $$
with $\varrho_n:= \frac{\pi}{2a} \sqrt{n^2\!-\!1}\;$, $\chi_n(u):= \frac{1}{\sqrt{a}}\sin\frac{\pi n}{2a}(u+a)$, and $(\cdot, \cdot)$ being the inner product in $L^2(-a,a)$, the sum runs in fact over even $n$ only.
\end{theorem}
\begin{proof}
In view of Theorem~\ref{theorem_spectrum_general} one has to establish the existence of a discrete spectrum for the Dirichlet Laplacian on the bent strip which can done using a variational argument, the idea of which belongs to Goldstone and Jaffe, cf.~\cite{GJ92} and \cite[Thm.~1.1]{EK15}. The asymptotic expansion in the mild-bending situation follows from \cite[Thm.~6.3]{EK15}.
\end{proof}

We mentioned that interesting spectral results can also be obtained for finite strips. A notable example is an isoperimetric-type inequality for loop-shaped strips of halfwidth $a$ build over smooth and closed curves without self-intersections of a \emph{fixed length} $L>0$. In that case the essential spectrum of $H_{m,\Omega}$ consists of a single point, the infinitely degenerate eigenvalue $mc^2$. The rest of the spectrum is purely discrete accumulating only at $\pm\infty$. Asking about optimization of the `smallest' pair of discrete eigenvalues, $\pm\lambda_1$, we get the following result:
\begin{theorem} \label{thm:optimstrip}
In this situation, $\lambda_1$ is uniquely maximized by $\Omega$ in the form of a circular annulus.
\end{theorem}
\begin{proof}
As the problem reduces to the maximization of the principal eigenvalue of $-\Delta_\mathrm{D}^\Omega$, the claim follows from Theorem~1, part (a), and the remark afterwards in \cite{EHL99}.
\end{proof}

Theorem~\ref{theorem_spectrum_general} allows us to get more information about the spectrum of $H_{m,\Omega}$ for infinite curved strips than stated above. For instance, an application of Payne-P\'olya-Weinberger inequality proved by Ashbaugh and Benguria \cite{AB92} yields a lower bound on the distance between $\lambda_1$ and the infinitely degenerate eigenvalue $mc^2$:
\begin{proposition} \label{prop:ppw2D}
Let $\Omega$ be a bent strip of halfwidth $a$ satisfying assumptions~\eqref{assa} and \eqref{assb} and suppose that $\#\sigma_\mathrm{disc}(H_{m,\Omega})=2N$, then we have
 \begin{equation} \label{ppw2D}
 \lambda_1 \ge c\sqrt{m^2c^2+3^{1-N}b_2\big(\textstyle{\frac{\pi}{2a}}\big)^2},
 \end{equation}
where $b_2:= \left(\frac{j_{0,1}}{j_{1,1}}\right)^2 \approx 0.394$ and $j_{r,1}$ is the first zero of the Bessel function $j_{r},\, r=0,1$.
\end{proposition}
The claim is a straightforward consequence of \cite[Thm.~3.1]{EK15}; it is clear that the bound is the strongest when $H_{m,\Omega}$ has just one pair of discrete eigenvalues. Similarly, one can translate to the Dirac operator setting other spectral estimates valid for Dirichlet Laplacians in strips, for instance, the Lieb-Thirring-type inequality for the moments of the sequence  $\big\{c\sqrt{\big(\frac{\pi}{2a}\big)^2+(mc)^2}-\epsilon_n\big\}\,$, where $\epsilon_n$ are the nonnegative eigenvalues of $H_{m,\Omega}$, cf.~\cite[Thm.~3.2]{EK15}.

Note further that the effect of geometrically induced binding we are discussing here is robust; it does not require the strip to have a smooth boundary. As an example, consider an L-shaped strip,
\begin{equation*}
  \Omega:= \{(x,y)\in\R^2: x,y>0,\, \min(x,y)<\pi\}.
\end{equation*}
For a Dirac operator on such a region, Theorem~1.2 and Proposition~1.2.3 of \cite{EK15} with $d=\pi$ imply the following claim:
\begin{proposition} \label{prop:Lshape}
$\sigma_\mathrm{ess}(H_{m,\Omega}) = (-\infty,-c\sqrt{m^2c^2+1}] \cup \{mc^2\} \cup [c\sqrt{m^2c^2+1},\infty)$ and the discrete spectrum consists of a pair of simple eigenvalues, $\pm c\sqrt{m^2c^2+\epsilon_1}$, where $\epsilon_1\approx 0.9291$.
\end{proposition}
In the same way, one can translate to the Dirac operator setting the results about spectra of more general polygonal ducts from Section~1.2 of \cite{EK15}.

Another situation where a local modification of a strip geometry may induce the existence of a discrete spectrum arises when the strip is locally modified as in~\eqref{weakly def strip}. A general existence result can be stated in a way independent of the cross section dimension and we state it in the next section, cf. Theorem~\ref{thm:locdef} below. In the two-dimensional situation we are able to demonstrate the following weak deformation behavior:
\begin{theorem} \label{thm:weaklocstripdef}
Let $\Omega = \Omega_{\beta f}$ be as in~\eqref{weakly def strip}. For small enough $\beta$, the discrete spectrum of $H_{m,\Omega}$ consists of a pair of simple eigenvalues, $\pm c\sqrt{m^2c^2+\epsilon_1(\beta)}$, provided $\langle f \rangle:= \int_\R f(x)\,\D x>0$. They are real-analytic functions at $\beta=0$ satisfying
 $$ 
 \epsilon_1(\beta)=  \Big(\frac{\pi}{d}\Big)^2 -\beta^2 \frac{\pi^4}{d^2} \langle f \rangle +\OO(\beta^3).
 $$ 
The discrete spectrum is empty if $\langle f \rangle<0$ as well in the critical case, $\langle f \rangle=0$, if $\beta$ is small and $8b< d\sqrt{3}$. On the other hand, a pair of weak bound states exists for $\langle f \rangle=0$ if
 $$ 
 \frac{\|f'\|^2}{\|f\|^2}< \frac{24}{9+\sqrt{117\!+\!48\pi^2}}\, \frac{\pi^2}{d^2}
 $$ 
and there are positive constants $c_1,\,c_2$ such that
 $$ 
 c_1\beta^4 \le \Big(\frac{\pi}{d}\Big)^2 - \epsilon_1(\beta) \le c_2\beta^4.
 $$ 
\end{theorem}
\begin{proof}
The results are obtained by Theorem~\ref{theorem_spectrum_general} and a variational method. The first claim follows from Theorem~6.5\footnote{This result in \cite{EK15} contains a misprint, $\langle f\rangle$ has there an extra square.} of \cite{EK15} which reproduces the result obtained in \cite{BGRS97}, the behavior in the critical case comes from \cite[Theorem~6.9]{EK15}.
\end{proof}

For laterally coupled strips described at end of Sec.~\ref{ss:2D} we have the following result:

\begin{theorem} \label{lateral bs}
We have $\sigma_\mathrm{disc}(H_{m,\Omega}) \ne\emptyset$ for any $\ell>0$. The eigenvalues $\pm c\sqrt{m^2c^2+\epsilon_j(\ell)}$ are simple with $\epsilon_j(\ell) \in \big(\big(\frac{\pi}{D}\big)^2, \big(\frac{\pi}{d}\big)^2 \big)$ which are continuously decreasing functions of $\ell$; their number, of both the positive and negative ones, is $2\left(\max\left\{1, \left\lfloor {\ell\over d} \sqrt{1\!-\!(1\!+\!\varrho)^{-2}} \right\rfloor \right\} + N\right)$ with $N \in \{0,1\}$. As for the weak coupling case, there are positive constants $c_1,\,c_2$ such that
 $$ 
 c_1 \ell^4 \le \big(\textstyle{\frac{\pi}{d}}\big)^2- \epsilon_1(\ell) \le c_2 \ell^4
 $$ 
holds for all sufficiently small positive window widths $\ell$.
\end{theorem}
\begin{proof}
The first claim follows from Theorem~1.5 of \cite{EK15}, the same result allows us to estimate positions of the eigenvalues and the critical values of $\ell$ at which eigenvalues emerge from the continuum. The weak coupling asymptotics is a consequence of Theorem~6.10 of \cite{EK15}.
\end{proof}

Before leaving the two-dimensional case, let us mention one more interesting example. This $\Omega$ consists of two strips, for simplicity both of the width $\pi$, \emph{crossing at the right angle}. In analogy with Proposition~\ref{prop:Lshape} one can check that $H_{m,\Omega}$ has then a pair of simple eigenvalues $\pm c\sqrt{m^2c^2+\epsilon_1}$, where $\epsilon_1\approx 0.66$. A new feature here is the existence of a pair of eigenvalues $\pm c\sqrt{m^2c^2+\epsilon_2}$ with $\epsilon_2\approx 2.73$ which are embedded in $\sigma_\mathrm{cont}(H_{m,\Omega})$ as one can check using a simple argument combing scaling and symmetry considerations, first proposed in \cite{SRW89}.

\section{The geometrically induced spectrum in three dimensions}\label{s:3Dspectrum}
\setcounter{equation}{0}

Let us pass to the three-dimensional situation, starting again from the essential spectrum of the Dirac particle confined to a tube or layer.
\begin{theorem} \label{thm:spectess3D}
We have $\sigma_\mathrm{ess}(H_{m,\Omega}) = (-\infty,-\epsilon_\mathrm{t}] \cup \{mc^2\} \cup [\epsilon_\mathrm{t},\infty)$, where
 \begin{enumerate}[(i)]
  \setlength{\itemsep}{0pt}
\item $\epsilon_\mathrm{t} = c\sqrt{m^2c^2+\mu_1}$, where $\mu_1$ is the principal eigenvalue of the Dirichlet Laplacian $-\Delta_\mathrm{D}^{M}$, if $\Omega$ is a bent tube of cross section $M$ satisfying assumptions~\eqref{assc}--\eqref{asse}.
\item The same is true if $\Omega$ is a locally deformed tube satisfying assumptions~\eqref{assf} and \eqref{assg}.
\item The same is true for the twisted tube \eqref{twisttube} provided the derivative $\dot\alpha$ is compactly supported.
\item If, on the other hand, the tube is periodically twisted outside a compact region, $\dot\alpha(x_1)=\beta$ for all $|x_1|$ large enough, we have $\epsilon_\mathrm{t} = c\sqrt{m^2c^2 + \inf \sigma(-\Delta_\mathrm{D}^{M} -\beta^2\pd^2_\varphi)}$, where $\varphi$ is the polar angle associated with the transverse variable $x_\perp$.
\item $\epsilon_\mathrm{t} = c\sqrt{m^2c^2+(\frac{\pi}{d})^2}$ if $\Omega$ is a curved layer of width $d=2a$ satisfying assumptions~\eqref{assh}, \eqref{assi}, and, in addition, $(h_\mu^\nu)_{,\rho}$ and $k_g:=\dot rr^{-1}$ vanish as $s\to\infty$.
\item The same is true for the bulged layer \eqref{bulgedlayer}.
\item For laterally coupled layers as in~\eqref{laterally_coupled_3D} of the widths $d_1,\,d_2>0$ we have $\epsilon_\mathrm{t} = c\sqrt{m^2c^2+(\frac{\pi}{d})^2}$, where $d=\max\{d_1,d_2\}$.
 \end{enumerate}
\end{theorem}
\begin{proof}
In view of Theorem~\ref{theorem_spectrum_general}, the claims follow from the essential spectrum properties of the corresponding Dirichlet Laplacians, specified in \cite[Prop.~1.3.1]{EK15} for claim (i), \cite[Thm.~1.4]{EK15} for claim (ii). If the tube is twisted, it is straightforward to check that $-\Delta_\mathrm{D}^\Omega$ is unitarily equivalent to $H_{\dot\alpha}:= -\Delta_\mathrm{D}^M - (\dot\alpha(x_1)\pd_\varphi+\pd_{x_1})^2$ on $L^2(\Omega_0)$, cf.~\cite[Sec.~1.7.1]{EK15}. If $\dot\alpha$ is compactly supported, this operator acts as $-\Delta_\mathrm{D}^{\Omega_0}$ outside a compact set and claim (iii) follows. If the tube is periodically twisted, $H_{\dot\alpha}$ is unitarily equivalent to the direct integral $\int^\oplus_\R h(p)\,\D p$, where
 \begin{equation} \label{dirint}
 h(p) = -\Delta_\mathrm{D}^M + (p-i\beta\pd_\varphi)^2
 \end{equation}
and $\sigma(H_{\dot\alpha}) = \sigma_\mathrm{ess}(H_{\dot\alpha}) = [\inf\sigma(h(0)),\infty)\,$ \cite[Prop.~1.7.3]{EK15}. The essential spectrum is preserved under compactly supported perturbations which gives claim (iv).

For curved layers satisfying assumptions~\eqref{assh} and \eqref{assi} we know from \cite[Prop.~4.2.1]{EK15} that $\inf\sigma_\mathrm{ess}(-\Delta_\mathrm{D}^\Omega)=(\frac{\pi}{2a})^2$; for that it is, in fact, sufficient for $\Sigma$ to be $C^2$-smooth. To complete the proof of claim (v) by showing that the essential spectrum covers the interval $[(\frac{\pi}{2a})^2,\infty)$ one can use Weyl's criterion in the way described in \cite{Kr01} and Remark~4.1.1 of \cite{EK15}; this requires the mentioned additional assumptions. For claims (vi) and (vii) we can similarly refer to the proof of Theorem~4.5 in \cite{EK15} and to \cite[Thm.~4.7]{EK15}, respectively.
\end{proof}

What is again more interesting is the discrete spectrum of $H_{m,\Omega}$ induced by the geometry of $\Omega$. If it is nonvoid, by Theorem~\ref{theorem_spectrum_general} it is mirror-symmetric, $\sigma_\mathrm{disc}(H_{m,\Omega})=\{\pm\lambda_j\}$, where we suppose the positive part to be arranged in the ascending order, multiplicity included.
\begin{theorem} \label{thm:benttube}
Let $\Omega$ be a bent tube of cross section $M$ satisfying assumptions~\eqref{assc}--\eqref{asse}, then $\sigma_\mathrm{disc}(H_{m,\Omega}) \ne\emptyset$. Furthermore, let $\{\Omega_\beta\}$ be a family of tubes with the cross section $M$ and torsion $\tau$ fixed and the curvature equal to $\beta\gamma$ for a fixed function $\gamma$ consistent with \eqref{assc} such that $\gamma, \dot\gamma, |\ddot\gamma|^{1/2} \in L^1(\mathbb{R}, |s| \D s)$; then for all $\beta$ small enough $H_{m,\Omega_\beta}$ has just two discrete eigenvalues $\pm\lambda_1(\beta)$, each of multiplicity two, such that
 \begin{equation} \label{weakbenttube}
 \lambda_1(\beta) = c\sqrt{m^2c^2+\epsilon(\beta)},
 \end{equation}
where
 \begin{align*} \label{}
 \sqrt{\mu_1-\epsilon(\beta)}\: =\: & \frac{\beta^2}{8}\|\gamma\|^2 - \frac{\beta^2}{16} \sum_{n=2}^\infty \int_{M\times M} \D y\D y'\, \chi_1(y)\chi_1(y')\chi_n(y)\chi_n(y') \\[.3em]
 & \times \sqrt{\mu_n-\mu_1} \int_{\R^2} h_s(s,y)\,\ee^{-\sqrt{\mu_n-\mu_1}|s-s'|}\,h_s(s',y')\, \D s\D s' + \OO(\beta^3).
 \end{align*}
In this expression $\mu_n$ are the eigenvalues of $-\Delta_\mathrm{D}^{M}$ and $\chi_n$ are the corresponding normalized eigenfunctions; the function $h_s:= r\gamma\tau\sin(\theta-\alpha) + r\dot\gamma\cos(\theta-\alpha)$ is integrated over $\D y= r\D r\D\theta$.
\end{theorem}
\begin{proof}
The existence of a nontrivial discrete spectrum of $H_{m, \Omega}$ follows by Theorem~\ref{theorem_spectrum_general} from the corresponding property of $-\Delta_\mathrm{D}^{\Omega}$, which can be checked by a variational method \cite[Thm.~1.3]{EK15}, the weak-coupling expansion of the principal eigenvalue is obtained using the Birman-Schwinger principle \cite[Thm.~6.3]{EK15}.
\end{proof}

Proposition~\ref{prop:ppw2D} has a three-dimensional analogue:
\begin{proposition} \label{prop:ppw3D}
Let $\Omega$ be a bent tube of cross section $M$ satisfying assumptions~\eqref{assc}--\eqref{asse} and suppose that $\#\sigma_\mathrm{disc}(H_{m,\Omega})=4N$, then we have
 \begin{equation} \label{ppw3D}
 \lambda_1 \ge c\sqrt{m^2c^2+3^{1-N}b_3\mu_1},
 \end{equation}
where $b_3:= \left(\frac{\pi}{j_{3/2,1}}\right)^2 \approx 0.489$ with $j_{3/2,1}$ being the first zero of the Bessel function $j_{3/2}$.
\end{proposition}
The result again comes from \cite[Thm.~3.1]{EK15}. Other spectral features of Laplacians such as the Lieb-Thirring-type inequality for the moment of the sequence $\{\mu_1-\epsilon_n\}\,$ \cite[Thm.~3.2]{EK15}, where $\epsilon_n$ are the nonnegative eigenvalues of $-\Delta_\mathrm{D}^\Omega$, translate to the Dirac operator setting as well.

Passing from curved strips and tubes to straight but locally deformed ones, one can adapt for our purpose Theorem~1.4 of \cite{EK15}:
\begin{theorem} \label{thm:locdef}
Let $\Omega$ be the locally deformed tube \eqref{deftube} satisfying~\eqref{assf} and~\eqref{assg}. The discrete spectrum of $H_{m,\Omega}$ is empty if $M_x\subset M$ for all $x\in\R$. On the other hand, if $M_x\supset M$ for each $x\in\R$ and there is an interval where $M_x\setminus M$ has a nonzero measure, then $\sigma_\mathrm{disc}(H_{m,\Omega}) \ne\emptyset$.
\end{theorem}
No general existence claim can be made if the deformation of a three-dimensional tube is `sign changing', however, in analogy with the two-dimensional case of Theorem~\ref{thm:weaklocstripdef} one naturally conjectures that for gentle deformations the positivity of the added volume will be decisive for the existence of a mirror pair of weakly bound states of $H_{m,\Omega}$.

Let us turn to twisted tubes. In contrast to bending, twisting gives rise to effective repulsion, and as such stabilizes the spectrum. Theorem~1.8 of \cite{EK15} yields the following result which is useful to compare to Theorem~\ref{thm:benttube} above:
\begin{theorem} \label{thm:stabspect}
Let $\Omega$ satisfy assumption \eqref{assc} and \eqref{asse}. Assume further that the cross section $M$ has a $C^2$ boundary and is \emph{not} a disc centered at the origin. Moreover, suppose that $\alpha$ in \eqref{curvilin} is a continuously differentiable function which \emph{violates the condition} $\dot\alpha=\tau$ such that $\dot \alpha$ is compactly supported, not identically equal to zero and has a bounded derivative. Then there is an $\varepsilon>0$ such $\|\gamma\|_\infty + \|\dot\gamma\|_\infty<\varepsilon$ implies $\sigma_\mathrm{disc}(H_{m,\Omega}) =\emptyset$.
\end{theorem}

By Theorem~\ref{thm:spectess3D}(iv) a periodic twist changes the essential spectrum of $H_{m,\Omega}$. Given the fact that the effective repulsion grows stronger with the twist velocity $\beta$, one may expect that a local change of the twist of the correct sign could give rise to the existence of bound states. This is indeed the case. Consider a straight twisted tube \eqref{twisttube} with of the cross section $M$ which is not a disc centered at the origin, has a $C^2$-boundary, and moreover,
 \begin{equation} \label{loctwist}
 \dot\alpha(x_1) = \beta - \delta(x_1),
 \end{equation}
where $\delta:\R\to\R$ is a bounded function supported in an interval $[-a,a]$ for some $a>0$.
\begin{theorem} \label{thm:loctwist}
In the described situation, $\sigma_\mathrm{disc}(H_{m,\Omega}) \ne\emptyset$ holds provided
 \begin{equation} \label{negtwist}
 \int_\R \big(\dot\alpha^2(x_1) - \beta^2)\,\D x_1 < 0.
 \end{equation}
Moreover, the claim remains valid even when the integral \eqref{negtwist} vanishes provided that $\dot\alpha(x_1) + \beta>0$ holds whenever $|x_1|\le a$ and $\ddot\alpha\in L^2(-a,a)$.
\end{theorem}
\begin{proof}
The two claims follow respectively from Theorems~1.9 and 1.10 in \cite{EK15}.
\end{proof}

Let us pass to operators $H_{m,\Omega}$ supported by layers. If $\Omega$ is curved but asymptotically planar, we have the following result:
\begin{theorem} \label{thm:curvedlayer}
Let $\Omega$ be a layer of halfwidth $a$ satisfying assumptions~\eqref{assh}--\eqref{assj}, then $\sigma_\mathrm{disc}(H_{m,\Omega}) \ne\emptyset$ holds if one of the following conditions is satisfied:
 \begin{enumerate}[(i)]
  \setlength{\itemsep}{0pt}
\item The total Gauss curvature satisfies $\KK\le 0$. This happens, in particular, if the generating surface $\Sigma$ is not conformally equivalent to the plane so that $\Omega$ is not simply connected.
\item The halfwidth $a$ is small enough and $\nabla_g M\in L^2_\mathrm{loc}(\Sigma)$, where $\nabla_g$ refers to covariant derivatives on the manifold $(\Sigma,g)$.
\item $\MM=\infty$ and $\nabla_g M\in L^2(\Sigma)$.
\item $\Sigma$ has a cylindrical symmetry; if $\KK>0$ we have $\#\sigma_\mathrm{disc}(H_{m,\Omega})=\infty$.
 \end{enumerate}
\end{theorem}
\begin{proof}
The first three claims follow from \cite[Thm.~4.2]{EK15}. In particular, the first one provides a universal existence result for layers which are not simply connected. This follows from Cohn-Vossen inequality, $\KK\le2\pi(2-2h-e)$, where $e$ is the number of ends of $\Sigma$ and $h$ is its genus, i.e. the number of `handles'; once it is nonzero, we have always $\KK<0$, cf. \cite[Corollary~4.2.1]{EK15}. The last claim is a consequence of Theorems~4.3 and~4.4 in \cite{EK15} and a corollary of the former.
\end{proof}
\begin{remark}
{\rm The trial functions used to establish claim (iv) of the last theorem have compactly supports which can be chosen arbitrarily far from the symmetry axis. This shows that the result remains valid if such a cylindrically layer is locally deformed. An example is a \emph{conical layer}, cf.~\cite[Example~4.2.3]{EK15} and \cite{ET10}.}
\end{remark}
The claim (iv) shows an important difference between curvature-induced bound states in tubes and layers. The former case is of a local character, while for layers the global geometry plays a role. This is connected with the fact that a tube can fully be `straightened' using curvilinear coordinates \eqref{curvilin}, while in layers the metric tensor of $\Sigma$ remains always present. It also means that if we want a weak-coupling result analogous to the last claim of Theorem~\ref{thm:benttube}, we have to restrict our attention to locally curved layers for which $\KK=0$. Consider, for instance, the family of surfaces
 \begin{equation} \label{wekdefsurf}
 \Sigma_\beta:= p(\R^2),\qquad p(x;\beta) = \big(x,\beta f(x)\big),
 \end{equation}
where $p:\R^2\to\R$ is a given $C^4$-smooth function. For simplicity we suppose that $f \not\equiv 0$ is of a compact support, the conclusions extend to situations where $f$ together with its derivatives up to the fourth order has suitable decay properties at infinity \cite[Thm.~6.4]{EK15}. We have the following claim:
\begin{theorem} \label{thm:curvlayer}
Let $\{\Omega_\beta\}$ be a family of layers of halfwidth $a$ built over the surfaces \eqref{wekdefsurf}. Then $H_{m,\Omega_\beta}$ has for all $\beta$ small enough exactly one pair of discrete eigenvalues $\pm\lambda_1(\beta)$, each having multiplicity two, behaving as
 \begin{equation} \label{curvlayer}
 \lambda_1(\beta) = c\sqrt{m^2c^2+\big(\textstyle{\frac{\pi}{2a}}\big)^2 -\ee^{2w(\beta)^{-1}}},
 \end{equation}
where
 $$ 
 w(\beta) = -\beta^2 \sum_{n=2}^\infty (\chi_1, u\chi_n)^2 \Big(\frac{\pi}{2a}\Big)^4 (n^2-1)^2 \int_{\R^2} \frac{|\hat m_0(\omega)|^2}{|\omega|^2 + \big(\frac{\pi}{2a}\big)^2 (n^2-1)}\, \D\omega + \OO(\beta^3)\,;
 $$ 
here $\{\chi_n\}$ are the normalized eigenfunctions of the Dirichlet Laplacian on $(-a,a)$, moreover, $(\cdot, \cdot)$ is the inner product in $L^2(-a,a)$ and $\hat m_0$ is the Fourier image of $m_0:=\frac12\Delta f$.
\end{theorem}
\begin{remark}
{\rm Returning to the dimensional consideration of Remark~\ref{rm:dim} we note that $w(\beta)$ is dimensionless as it should be. On the other hand, Theorem~6.4 of \cite{EK15} on which we rely is again stated as a mathematical result; properly speaking the last term in the square root of \eqref{curvlayer} should read $-\frac{1}{L^2}\,\ee^{2w(\beta)^{-1}}$, where $L$ is the quantity fixing the length scale. This does not change the principal conclusion, namely that the weak-coupling behavior of these bound states is the same as for two-dimensional Schr\"odinger operators, with an exponential small gap. The same conclusions can be made about the weak-coupling gap in Theorems~\ref{thm:weaklocdef} and \ref{thm:weaklater} below.}
\end{remark}

In contrast to the two-dimensional situation little is known about spectral properties of sharply broken layers. Using the result of \cite{DLO18} we can make a claim about the `octant', or `Fichera' layer, which can be regarded as a counterpart to the L-shaped strip, namely the region
$\Omega:= \{(x,y,z)\in\R^2: x,y,z>0,\, \min(x,y,z)<\pi\}$.
\begin{proposition} \label{prop:octant}
One has $\sigma_\mathrm{ess}(H_{m,\Omega}) = (-\infty,-c\sqrt{m^2c^2+\epsilon_\infty}] \cup \{mc^2\} \cup [c\sqrt{m^2c^2+\epsilon_\infty},\infty)$, where $\epsilon_\infty\approx 0.93$ refers to the L-shaped planar strip of the width $\pi$ in Proposition~\ref{prop:Lshape}. The discrete spectrum of $H_{m,\Omega}$ consists at most of a finite number\footnote{In \cite{DLO18} a numerical argument is used to show that the discrete spectrum is nonempty.} of eigenvalues of the form $\pm c\sqrt{m^2c^2+\epsilon_j}$ with $\epsilon_j\in(0,\epsilon_\infty)$.
\end{proposition}

The isoperimetric-type inequality of Theorem~\ref{thm:optimstrip} has a three-dimensional analogue. Consider a layer of a fixed halfwidth $a$ built over a compact surface $\Sigma$ without a boundary. The essential spectrum of $H_{m,\Omega}$ consists again of a single point, the infinitely degenerate eigenvalue $mc^2$, and the rest of the spectrum is purely discrete accumulating only at $\pm\infty$. We take the family of all such layers of halfwidth $a$ satisfying assumption~\eqref{assh} and such that area of the surface $\Sigma$ is fixed, and ask about optimization of the `smallest' pair of eigenvalues, $\pm\lambda_1$.
\begin{theorem} \label{thm:optimlayer}
In this situation, $\lambda_1$ is uniquely maximized by layers built over a spherical $\Sigma$.
\end{theorem}
\begin{proof}
What matters is again the principal eigenvalue of $-\Delta_\mathrm{D}^\Omega$, and therefore, the claim follows from Theorem~1, part (b), in \cite{EHL99}.
\end{proof}

Returning to infinite layers, consider now locally deformed ones of the type \eqref{bulgedlayer} with a compactly supported function $f$. As before a local protrusion creates bound states:
\begin{theorem} \label{thm:weaklocdef}
Let $f\ge 0$. If there is an $\eta>0$ and an open $\WW\subset\R^2$ such that $f(x)>\eta$ for $x\in\WW$, $\:\sigma_\mathrm{disc}(H_{m,\Omega}) \ne\emptyset$. Moreover, if $f$ is replaced by $\beta f$ with $f$ not necessarily positive, but such that $\langle f\rangle:= \int_{\R^2} f(x)\,\D x>0$, the operator $H_{m,\Omega}$ has for all sufficiently small $\beta>0$ just one pair of eigenvalues, $\pm\lambda_1(\beta)$, each having multiplicity two, and
 \begin{equation} \label{weaklocdef1}
\lambda_1(\beta) = c\sqrt{m^2c^2+\big(\textstyle{\frac{\pi}{d}}\big)^2 -\ee^{2w(\beta)^{-1}}},\qquad
w(\beta) = -\beta\, \frac{\pi}{d^2}\,\langle f\rangle + \OO(\beta^2).
 \end{equation}
\end{theorem}
\begin{proof}
The claim follows from Theorems~4.5 and~6.6\footnote{This result in \cite{EK15} contains the same misprint, an extra square of $\langle f\rangle$.} of \cite{EK15}.
\end{proof}

Let us finally consider the laterally coupled layers. The essential spectrum of them is given by Theorem~\ref{thm:spectess3D}(vii). In this case the existence of a discrete spectrum is easy to establish but the weak-coupling result is less precise than in the previous cases:

\begin{theorem} \label{thm:weaklater}
Let $\WW$ be an open bounded set and let $\Omega$ be defined by~\eqref{laterally_coupled_3D}; whenever $\WW$ is nonempty, $\:\sigma_\mathrm{disc}(H_{m,\Omega}) \ne\emptyset$. Moreover, if $\WW=\beta M$ with $M$ open and nonempty, the operator $H_{m,\Omega}$ has for all sufficiently small $\beta>0$ just one pair of eigenvalues, $\pm\lambda_1(\beta)$, each having multiplicity two, and there are positive $c_1,\,c_2$ such
 \begin{equation} \label{weaklocdef}
\ee^{-c_2\beta^{-3}} \le \epsilon_\mathrm{t}-\lambda_1(\beta) \le \ee^{-c_1\beta^{-3}}.
 \end{equation}
\end{theorem}
\begin{proof}
The claim follows from \cite[Theorem~4.7]{EK15} and \cite[Theorem~3.1]{EV97}.
\end{proof}

\subsection*{Acknowledgements}

The research was supported by the Czech Science Foundation within the project 21-07129S and by the EU project CZ.02.1.01/0.0/0.0/16\textunderscore 019/0000778. M. Holzmann gratefully acknowledges financial support by the Austrian Science Fund (FWF): P 33568-N.

\end{document}